\documentclass[11pt]{amsart}
\usepackage{times}
\usepackage[english]{babel}
\usepackage[all]{xy}
\usepackage{amscd, amssymb, amsmath, amsthm, graphics}
\usepackage{amsmath,amsfonts,amsthm,amssymb}
\usepackage{latexsym,amsmath}
\usepackage{graphicx,psfrag}
\usepackage{pinlabel}
\usepackage{mathrsfs}

\newtheorem{theorem}{Theorem}[section]
\newtheorem{proposition}[theorem]{Proposition}
\newtheorem{corollary}[theorem]{Corollary}
\newtheorem{lemma}[theorem]{Lemma}

\theoremstyle{definition}

\theoremstyle{remark}

\newtheorem{question}{Question}
\numberwithin{equation}{section}

\renewcommand{\hat}{ \widehat}

\newcommand{\inte}{{\mathrm{int}}}

\renewcommand{\epsilon}{\varepsilon}

\begin{document}
\sloppy

\title[Invariant incompressible surfaces in reducible 3-manifolds]{Invariant incompressible surfaces in reducible 3-manifolds}

\author{Christoforos Neofytidis }
\address{Department of Mathematical Sciences, {\smaller SUNY} Binghamton, Binghamton NY 13902-6000, USA}
\email{chrisneo@math.binghamton.edu}
\author{Shicheng Wang}
\address{Department of Mathematics, Peking University, Beijing 100871, China}
\email{wangsc@math.pku.edu.cn}
\thanks{We would like to thank Professor Xuezhi Zhao for helpful conversations and comments. The first author gratefully acknowledges the hospitality and support of Peking University where part of this research was carried out.}


\date{\today}

\begin{abstract}
We study the effect of the mapping class group of a reducible 3-manifold $M$ on each incompressible surface that is invariant under a self-homeomorphism of $M$.
As an application of this study we answer a question of  F. Rodriguez Hertz, M. Rodriguez Hertz and R. Ures: 
A reducible 3-manifold admits an Anosov torus if and only if one of its prime summands is either the 3-torus, the mapping torus of $-id$, or the mapping torus of a hyperbolic automorphism. 
\end{abstract}

\maketitle
\vspace{-.5cm}

\section{Introduction}

 A closed oriented connected 3-manifold $M$ is called irreducible if every embedded 2-sphere in $M$ bounds a 3-ball; otherwise $M$ is called reducible. We say that an embedded closed oriented connected surface $F\neq S^2$ in $M$ is incompressible if whenever $\partial D\cap F= \partial D$ for an embedded disc $D\subset M$, then $\partial D$ bounds a disc in $F$; equivalently the homomorphism $\pi_1(F) \to \pi_1(M)$ induced by inclusion is injective.
 We refer to \cite{He} for standard notions and terminology on 3-manifolds.
 
If $M$ is a reducible 3-manifold, then by the Kneser-Milnor theorem it can be decomposed,
 uniquely up to diffeomorphism, into a finite connected sum 
 $$M = M_ 1 \# M_ 2  \# \cdots \# M_n\#(\#_m S^1\times S^2),$$
 where each $M_i$ is irreducible and $m\geq 0$. 
 
 The following fundamental result on the mapping class groups of reducible 3-manifolds was first announced in \cite{CR}. An elegant  proof of this theorem was given by McCullough in~\cite[pp. 70--71]{Mc}. As McCullough remarks, his proof is based on an argument of Scharlemann which appeared in~\cite[Appendix A]{Bon}. 
 

\begin{theorem} {\normalfont(\cite[page 69]{Mc}).} \label{Mc1} 
Let $M$ be a compact oriented connected 3-manifold. Then any orientation-preserving homeomorphism of $M$ is  isotopic to a composite of the following four types of homeomorphisms: 
\begin{itemize}
\item[(1)] homeomorphisms  preserving  summands;
\item[(2)] interchanges of homeomorphic summands;
\item[(3)] spins of $S^ 1 \times S^ 2$ summands;
\item[(4)] slide homeomorphisms.
\end{itemize}
\end{theorem}


In fact, the proof of Theorem \ref{Mc1} presented in \cite[pp. 70-71]{Mc} contains a finer form of that statement as follows:

\begin{theorem} \label{Mc2} 
Let $M$ be a compact oriented connected 3-manifold and $f$ be an orientation-preserving homeomorphism of $M$. Then 
$$hf  =g_3 g_2 g_1,$$ 
where $h$ is a finite composition of homeomorphisms of type 4 (slide homeomorphisms) and isotopies on $M$, and each $g_k$ is a composition of finitely many homeomorphisms of type $k$ on $M$. 
\end{theorem}

Theorem \ref{Mc2} already implies that $hf$ permutes the prime summands of $M$. The main result of this paper is that $h$ can be chosen so that its restriction on each $f$-invariant incompressible surface is the identity.

\begin{theorem}\label{main theorem}
Let $f\colon M\to M$ be an orientation-preserving homeomorphism of a closed oriented connected $3$-manifold $M$. If $F$ is an incompressible surface in $M$ with $f(F)=F$, then $h$ in Theorem \ref{Mc2} can be chosen to be the identity on $F$.
\end{theorem}

An immediate consequence of Theorem \ref{main theorem} and its proof is the following:

\begin{corollary}\label{preserve summands}
Suppose  $f\colon M\to M$ is an orientation-preserving homeomorphism of a closed oriented connected 3-manifold $M$ and $F$ an $f$-invariant incompressible surface.
Then $F$ can be isotoped into a prime summand of $M$ so that $hf\colon M\to M$ 
preserves this prime summand and $F$, 
where $h$ is a finite composition of slide homeomorphisms and isotopies.
\end{corollary}

Incompressible surfaces that are invariant under homeomorphisms play important roles in the study of 3-manifolds, in particular with respect to the effect of the induced action on their fundamental group. 
We say that an embedded 2-torus $T$ in a 3-manifold $M$ is an {\em Anosov torus} if there exists a diffeomorphism $f$ on $M$ such that $f(T)=T$ and the induced action of $f$ over the fundamental group of $T$ is hyperbolic, or equivalently $f|_T$ is (isotopic to) an Anosov map.
 
Motivated by problems in partially hyperbolic dynamics, F. Rodriguez Hertz, M. Rodriguez Hertz and R. Ures proved that only a few irreducible 3-manifolds admit Anosov tori:
 
 \begin{theorem}\label{HHU1}{\normalfont(\cite[Theorem 1.1]{HHU}).}
 A closed oriented connected irreducible 3-manifold admits an Anosov torus if and only if it is one of the following: 
 \begin{itemize}
\item[(1)] the 3-torus; 
\item[(2)] the mapping torus of $-id$; 
\item[(3)] the mapping torus of a hyperbolic automorphism.
\end{itemize}
 \end{theorem}
 
As pointed out in~\cite{HHU}, it is easy to construct arbitrarily many different reducible 3-manifolds that admit Anosov tori. Indeed, once a manifold $M$ admits an Anosov torus, then the connected sum of $M$ with any other 3-manifold admits an Anosov torus as well; see \cite[Remark 2.6]{HHU}.
Thus, in view of Theorem \ref{HHU1}, the following interesting question arises, which was our inspiration for this paper:
 
 \begin{question}{\normalfont(\cite[Question 1.4]{HHU}).}\label{q:HHU}
 Let $M$ be a closed oriented connected reducible 3-manifold. 
 If $M$ admits an Anosov torus, is one of its prime summands necessarily one of the 3-manifolds listed in Theorem \ref{HHU1}?
 \end{question}
 
 As an application of Theorem \ref{main theorem}, we answer Question \ref{q:HHU} in the affirmative:
 
\begin{theorem} \label{main1}
 A closed oriented connected 3-manifold admits an Anosov torus if and only if one of its prime summands is one of the following: 
\begin{itemize} 
\item[(1)] the 3-torus; 
\item[(2)] the mapping torus of $-id$; 
\item[(3)] the mapping torus of a hyperbolic automorphism.
\end{itemize}
 \end{theorem}

\subsection*{Outline} In Section \ref{mapping class group} we recall the descriptions of the  four types of homeomorphisms given in Theorem \ref{Mc1}. In Section \ref{proof of theorem} we prove Theorem \ref{main theorem}, Corollary \ref{preserve summands} and Theorem \ref{main1}.
 
 \section{Mapping class groups of reducible 3-manifolds}\label{mapping class group}
 
In this section we recall the isotopy types of the orientation-preserving homeomorphisms of $3$-manifolds given in Theorem \ref{Mc1}. We follow McCullough's survey paper \cite{Mc} for the description of the mapping class groups of reducible 3-manifolds. 
Suppose $M$ is a closed oriented reducible 3-manifold. By the Kneser-Milnor theorem, $M$ admits a non-trivial decomposition
$$M = M_ 1 \# M_ 2  \# \cdots \# M_n \# ( \#_m S ^1  \times S^2 ),$$    
where the summands  $M_i$ are irreducible and $m\ge 0$.  

Consider the following construction of $M$: Remove $n + 2m$ open 3-balls from a 3-sphere to  obtain a punctured 3-cell  $W$ with boundary components $S_1, S_2, ... , S_n, S_{n+1,0}, S_{n+1,1}, ..., S_{n+m,0}, S_{n+m,1}$. For each summand $M_i$, $i=1,...,n$, choose a 3-ball $D_i$ and attach $M_i'= M_i - \inte(D_i)$ to  $S_i$ along $\partial D_i$. For  $n+ 1\le j \le n+m$, let $S_j \times I$ be  a  copy of  $S^2\times I$  attached  to   $W$ by identifying $S_j \times   0$ with  $S_{j,0}$ and $S_{j} \times 1$ with $S_{j,1}$ to form an $S^1  \times  S^2$ summand.

Using the above construction, we now describe the four types of homeomorphisms of $M$ given in Theorem \ref{Mc1}. Note that two orientation-preserving homeomorphisms of $W$ are isotopic if and only if they induce the same permutation on the set of boundary components of  $W$. 

\medskip

{\em 1. Homeomorphisms  preserving  summands.} These are the homeomorphisms that restrict to the identity on $W$. They form a subgroup of $\mathrm{Homeo}(M)$
isomorphic  to  $\Pi_{i=1}^n \mathrm{Homeo}( M_i \ \text{ rel }   D_i) \times  \Pi_{j=1}^m  \mathrm{Homeo} (S^2  \times  I \   \text{ rel }  S^ 2 \times \partial I).$  
Note that  $\mathrm{Homeo} (S^2  \times  I \ \text{ rel } S^ 2 \times \partial I)$  has  two path components, that of the identity and that of a rotation about $S^2  \times  1/2$.

\medskip

{\em 2. Interchanges of homeomorphic summands.} Suppose $M_i$ and  $M_j$ are orientation-preserving homeomorphic summands. Then we can construct a homeomorphism of $M$ fixing all other summands, leaving $W$ invariant, and interchanging $M'_i$ and $M'_j$. 
Similarly, we can interchange two $S^1 \times S^2$ summands, leaving $W$ invariant.

\medskip

{\em 3. Spins of $S^ 1 \times S^ 2$ summands.} For each $n + 1\le  j\le  n + m$, we can construct a homeomorphism of $M$ fixing all other summands, leaving $W$ invariant, interchanging $S_{j,0}$  and $S_{j,1}$,  and restricting to an orientation-preserving homeomorphism that interchanges the boundary components of $S_j \times I$.

\medskip

{\em 4. Slide homeomorphisms.} For $i\le n$, let $\hat M$ be obtained from $M$ by replacing   $M_i$ with  a  3-cell  $E$. Let  $\alpha$ be an arc in $\hat M$ meeting $E$ only  in  its endpoints. Choose an isotopy  $J_t$ of $\hat M$ with $J_ 0 = id_{\hat M}$ and $J _1 |_E  = id_E$, so  that $J_t$ moves  $E$ around $\alpha$.  A {\em slide homeomorphism on $M$ that slides $M_i$  around $\alpha$} is a homeomorphism $h\colon M\to M$ defined by $h|_{M - M'_i} = J_1 |_{\hat M - E}$ and $h|_{M'_i}=  id|_{M'_i}$. Choosing a different $J_t$, it changes $h$ by an isotopy and, possibly, by a rotation about  $S_i$. Thus a choice of $\alpha$ might determine two isotopy classes of a slide homeomorphism. 

Similarly, we can slide either end of $S_j \times I$ around an  arc in  $M - S_j \times  (0,1)$. 

Note, finally, that if $\alpha_1$ and $\alpha_2$ are two arcs meeting $E$ only in their endpoints, and $\alpha$ is an arc representing the product of $\alpha_1$ and $\alpha_2$ in  
$\pi_ 1 (M - M'_i, S_i)$, then a slide of $M_i$ around $\alpha$ is isotopic  to a composite of slides around $\alpha_ 1$ and $\alpha_ 2$. Similarly for sliding ends of $S_j \times  I$'s. It follows that the subgroup of $\mathrm{Diff}(M)$ generated by slide homeomorphisms is finitely generated.

\section{Controlled slidings and isotopies}
\label{proof of theorem}

We now prove Theorem \ref{main theorem}. Given a homeomorphism $f$ on $M$, we will perform slide homeomorphisms and isotopies, controlling their effect on each $f$-invariant incompressible surface.

 \begin{proof}[Proof of Theorem \ref{main theorem}]
 Suppose $F$ is an incompressible surface in $M$ and $f\colon M\to M$ is a homeomorphism so that $f(F)=F$.  

\medskip

The proof of the following lemma is straightforward:

\begin{lemma}\label{phi(F)}
For any homeomorphism $\varphi\colon M\to M$, $\varphi(F)$ is an $\varphi f\varphi^{-1}$-invariant incompressible surface in $M$.
\end{lemma} 
 
Let $\Sigma$ be the union of the $n+2m$ prime decomposition 2-spheres as described in Section \ref{mapping class group}.
By a standard argument in 3-manifold topology, there is an isotopy $H_t$ of $M$ so that $H_1(F)$ is disjoint from $\Sigma$. 
Thus we can assume that $F\cap \Sigma=\emptyset$, and so $F$ lies in some $M'_i$, say in $M'_1$ (possibly after replacing $f$ by $H_1fH_1^{-1}$ and $F$ by $H_1(F)$, according to Lemma \ref{phi(F)}). 
  It is then clear that $M_1$ is not one of the $S^2\times S^1$ summands. 
  Also, $F\cap \Sigma=\emptyset$ implies that $f(F)\cap f(\Sigma)=\emptyset$. Since $f(F)=F$, we have both 
 $$F\cap \Sigma=\emptyset \ \text{ and } \ F\cap f(\Sigma)=\emptyset. \qquad (1)$$
 
We may assume that $\Sigma$ and $f(\Sigma)$ meet transversely. The major part of the proof of Theorem \ref{Mc2} in \cite{Mc} is to modify $f$ to reduce the number of components of  $\Sigma\cap f(\Sigma)$ by a sequence of slide homeomorphisms and isotopies 
 so that  $\Sigma\cap f(\Sigma)=\emptyset$, and finally make a further isotopy so that   $\Sigma=f(\Sigma)$. 
 And $h$ in Theorem \ref{Mc2}  is the composition 
 of those slides and isotopies. 
 
 \begin{proposition}\label{not touching F}
Each factor of $h$ can be chosen so that it does not touch $F$ during its sliding/isotopy process.  
\end{proposition}
 
\begin{proof}[Proof of Proposition \ref{not touching F}]
We are going to prove this claim by examining at each step in the proof of Theorem \ref{Mc2} (Theorem \ref{Mc1}) the effect of our choice of $h$ on $F$.

Let $C$ be a circle of intersection that is innermost on $f(\Sigma)$, so that $C$ bounds a disc $ E_1$ in $f(\Sigma)$ with $\inte(E _1)$ disjoint from $\Sigma$. 

\medskip

{\em Case 1.}
If $E_ 1\subset  M_i$, then $E_1$ separates a 3-ball $B$ from $M_i'$ (since $M_i$ is a prime factor) and $\partial B=E_1\cup E'$, where $E'$ is a disk in $S_i-C$. Since $E_1\subset f(\Sigma)$ and $E'\subset \Sigma$, we conclude by (1) that $F\cap \partial B=\emptyset$. Since $F$ is incompressible, we indeed have $F\cap B=\emptyset$. We choose  a regular neighborhood $N(B)$ of $B$ such that $F\cap N(B)=\emptyset$.
Then there is an isotopy $s$ pulling $E_ 1$ into $W$ across $B$ with support in $N(B)$, eliminating $C$ (and possibly other circles of intersection as well). Clearly this isotopy process does not touch $F$, therefore (1) still holds when we replace $f$ with $sf$. For simplicity, we will continue using $f$ to denote $sf$. We call the isotopy that we just performed a {\em controlled isotopy}; see Figure \ref{f:controlledisotopy}.

Similarly, if $E_1\subset S_j\times I$, $n + 1 \leq j \leq n + m$,
then $C$ can be eliminated by a controlled isotopy which does not affect $F$. 

\begin{figure}
\labellist
\pinlabel $B$ at 65 105
\pinlabel $E_1$ at 31 105
\pinlabel $M'_i$ at 25 135
\pinlabel $E'$ at 101 105
\pinlabel $S_i$ at 55 173
\pinlabel $S_k$ at 245 123
\pinlabel $M'_k$ at 221 71
\endlabellist
\centering
\includegraphics[width=9cm]{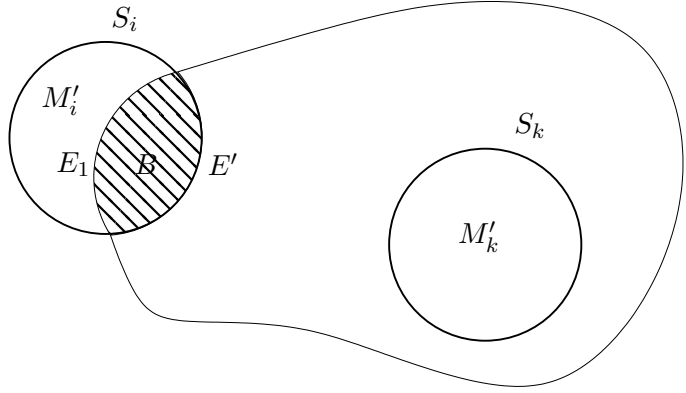}
\caption{\small A controlled isotopy pulls $E_1$ into $W$ eliminating $C$.}
\ \newline \
\label{f:controlledisotopy}
\end{figure}

\medskip

{\em Case 2.} 
 Suppose now  $E_ 1\subset W$ and $C\subset S_i$ for some $1\leq i\leq n$. We may assume that the interior of the other disc bounded by $\partial E_1$ in $f(\Sigma)$ intersects $\Sigma$, otherwise it must lie in $M'_i$ and so $C$ can be eliminated by a controlled isotopy as in Case 1. 
 Thus we can choose an arc $\alpha_0$ in $f(\Sigma)\cap M_i'$ with one endpoint in $C$ and the other endpoint in $S_i-C$. Denote by $E_2$  the disc in $S_i$  which is the closure of the component of $S_i-C$ that does not contain the other endpoint of $\alpha_0$.   
The  2-sphere $E_1 \cup E_2$ bounds a punctured 3-cell $W_1 \subset W$. Suppose $M_k$ (where $k\ne i$) is attached to $W_1$. There is an arc $\alpha$ with
endpoints in $S_k$ which consists of three portions: First $\alpha$ travels in $W_1$ to $S_i$,  then it goes through $M'_i-f(\Sigma)$ emerging in $W- W_1$, following along near $\alpha_0$, and finally it travels through $W$ back to $S_k$; see Figure \ref{f:controlledslidingsupport} (or Figure 1 in \cite{Mc}).  
 
Since $F\cap f(\Sigma)=\emptyset$ and  $\alpha_0\subset f(\Sigma)$, we can choose the second portion of $\alpha$ close enough to $\alpha_0$ so that it does not touch $F$.
Moreover, the first portion and the third portion of $\alpha$ also do not touch $F$, since they stay in $W$ and meet the second portion at its end points, and $F$ lies in $M'_1$. So we conclude that $\alpha$ does not touch $F$. By (1) we have 
$F\cap (S_k\cup \alpha)=\emptyset$. Therefore we can further find a regular neighborhood $N (S_k\cup \alpha)$ (the region bounded by bold lines in Figure \ref{f:controlledslidingsupport}) such that 
$F\cap N(S_k\cup \alpha)=\emptyset$. Now we quote the following fact whose proof follows rather directly from the definition, and which has been carefully presented in \cite{Z} with a fine figure illustration and precise formulas:

\begin{lemma}\label{slide}
 If $s$ is a slide homeomorphism along an arc $\alpha$ with ends in $S_k$, then $s$ is supported in a regular neighborhood $N(S_k\cup\alpha)$.
 \end{lemma}

\begin{figure}
\labellist
\pinlabel $\alpha_0$ at 87 155
\pinlabel $\alpha$ at 179 133
\pinlabel $E_1$ at 350 161
\pinlabel $M'_i$ at 95 109
\pinlabel $E_2$ at 139 99
\pinlabel $W_1$ at 163 75
\pinlabel $S_i$ at 101 169
\pinlabel $S_k$ at 289 99
\pinlabel $M'_k$ at 265 65
\pinlabel $N(S_k\cup\alpha)$ at 124 42
\endlabellist
\centering
\includegraphics[width=11cm]{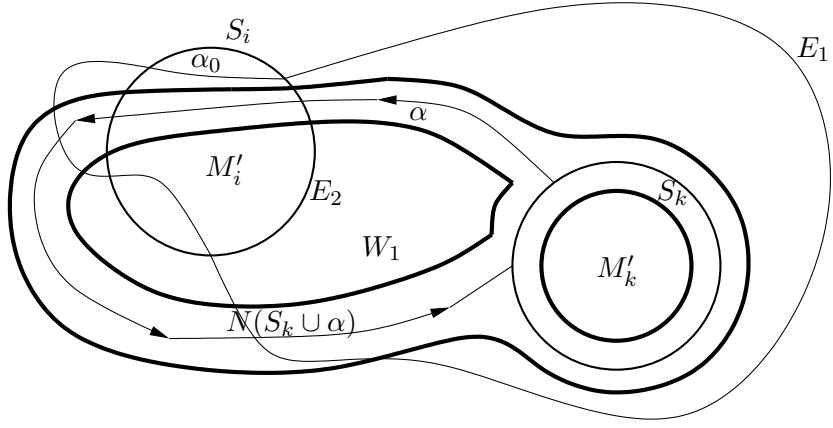}
\caption{\small A controlled sliding along $\alpha$ supported in $N(S_k\cup\alpha)$.}
\label{f:controlledslidingsupport}
\end{figure}

Slide $M_k$ around  $\alpha$, that is compose $f$ with the slide homeomorphism $s$ that slides $M_k$ around  $\alpha$, to reduce a puncture in $W_1$. By Lemma \ref{slide} and the fact that $F\cap N(S_k\cup \alpha)=\emptyset$, we have that the sliding $s$ does not touch $F$. Therefore (1) still holds when we replace $f$ with $sf$.
Again, for simplicity we still use $f$ to denote $sf$. We call the slide homeomorphism that we just performed a {\em controlled sliding} (see Figure \ref{f:controlledslidingsupport}).

Similarly, we slide each end of an $S_j\times I$ attached to $W_1$ without touching $F$.

We continue performing controlled slidings for each $M_k$ attached to $W_1$ and each end of an $S_j\times I$ attached to $W_1$ until we reach at the point where $E_1 \cup  E_2$ bounds a 3-ball in $W$. Now $C$ can be eliminated by a controlled isotopy. 


Finally, suppose $C\subset S_{j,0}$, for some $n+1\leq j\leq n+m$. If there is no arc in $f(\Sigma)\cap(S_j\times I)$ with one end in $C$ and the other end in $S_{j,1}$, then $C$ can be eliminated by a controlled isotopy. 
If there is an arc in $f(\Sigma)\cap(S_j\times I)$ with one end in $C$ and the other end in $S_{j,1}$, then we choose $E_2$ so that $S_{j,1}$ is not a boundary component of $W_1$ and slide as above summands $M_k$ attached to $W_1$ and each end of an $S_k\times I$ attached to $W_1$ until $C$ is eliminated. 

\medskip

Repeating the above controlled slidings and controlled isotopies as far as needed, we reach $f(\Sigma)\cap \Sigma =\emptyset$.

\medskip

Note that no component of $\Sigma$ (resp. of $f(\Sigma)$) can bound a 3-ball in $M$. For each  $M_i'$, there are two cases (similarly for each $S_j\times I$):

\begin{itemize}
\item[$\bullet$]  There is some $f(S_k)\subset M_i'$. Since $M_i$ is a prime factor, each 2-sphere in $M'_i=M_i - \inte(D_i)$ which does not bound a 3-ball must be parallel to $S_i$, that is to say,  $f(S_k)$ and $S_i$ bound a 3-manifold homeomorphic to $S^2\times [0,1]$ in $M$. 
\item[$\bullet$]  $M_i'\subset f(M_k')$ for some $k$. By the same reason as above, $f(S_k)$ must be  parallel to $S_i$. 
\end{itemize}

Both cases imply that $\Sigma$ and 
$f(\Sigma)$ bound a $\Sigma \times [0,1]$, a disjoint union of $n+2m$ copies of $S^2\times [0,1]$, in $M$.
Clearly we can isotope $f(\Sigma)$ to $\Sigma$ by pushing $f(\Sigma)$ across the $\Sigma \times [0,1]$ and then reach $f(W) = W$.
Since (1) holds and $F$ is incompressible, we have $F\cap N(\Sigma \times [0,1])=\emptyset$ for some regular neighborhood of $\Sigma \times [0,1]$, therefore the last isotopy can be made without touching $F$. 
This finishes the proof of Proposition \ref{not touching F}.
\end{proof}
 
We have now completed the proof of Theorem \ref{main theorem}.
 \end{proof}

 \begin{proof}[Proof of Corollary \ref{preserve summands}]
 Suppose $F$ is an incompressible surface in $$M=M_1\#\cdots\# M_n\#(\#_m S^2\times S^1),$$ where each $M_i$ is irreducible and $m\geq 0$. As explained in the proof of Theorem \ref{main theorem} (cf. Lemma \ref{phi(F)}), we may assume that $F$ lies in $M'_1=M_1-\inte(D_1)$.  
 
 Let $h$ be a composition of controlled slidings and isotopies as performed in Theorem \ref{main theorem} so that 
 $$hf({\Sigma})=g_3g_2 g_1({\Sigma})={\Sigma}. \qquad$$
 This means that $hf$ permutes the prime factors of $M$. Since moreover $h$ restricts to the identity on $F$ and $f(F)=F$, we deduce that
 $hf(F)=f(F)=F$. Since $F\subset M_1'$ we conclude that 
 $$hf(M'_1)=M_1'.$$
 \end{proof}

 \begin{proof}[Proof of Theorem \ref{main1}]
Let $M$ be a closed oriented connected $3$-manifold and suppose $T\subset M$ is an Anosov torus, that is, there is a diffeomorphism $f\colon M\to M$ such that $f(T)=T$ and the induced homomorphism $f_*\colon\pi_1(T)\to\pi_1(T)$ is hyperbolic.

Since every Anosov torus is incompressible \cite{HHU2}, we can assume that $T$ lies in some $M'_i$, say in $M'_1$ (cf. Lemma \ref{phi(F)}). By Theorem \ref{main theorem} and Corollary \ref{preserve summands} there is a finite  composition $h$ of slide homeomorphisms and isotopies so that
 $$hf(T)=T=f(T) \ \text{ and } \ hf(M'_1)=M_1'.$$
 
 We extend $hf|_{M_1'}$ to a diffeomorphism $g\colon M_1\to M_1.$ We have
$$g|_T=hf|_T=f|_T,$$
and so $g|_{T}\colon T \to T$ is Anosov. 
Since $M_1$ is irreducible, $M_1$ must be one of the three types of 3-manifolds listed in 
Theorem \ref{HHU1}. This finishes the proof.
\end{proof}

\end{document}